\documentclass[reqno]{amsart}

\usepackage{hyperref}
\usepackage{amsmath,amsthm,amssymb,amscd}
\usepackage{tikz}

\usetikzlibrary{intersections,calc}

\newtheorem{thm}{Theorem}

\newtheorem{cor}[thm]{Corollary}
\newtheorem{lem}[thm]{Lemma}
\newtheorem{clm}[thm]{Claim}
\newtheorem{comp}[thm]{Complement}

\theoremstyle{definition}
\newtheorem{dfn}[thm]{Definition}
\newtheorem{rem}[thm]{Remark}
\newtheorem{ex}[thm]{Example}

\newtheorem{ques}[thm]{Question}
\newtheorem{conj}[thm]{Conjecture}

\numberwithin{thm}{section}
\numberwithin{equation}{section}

\DeclareMathOperator{\diam}{diam}

\newcommand{\step}[1]{\medskip\noindent\textit{#1.}}

\title{Alexandrov spaces are CS sets}
\author[T. Fujioka]{Tadashi Fujioka}
\address{Department of Mathematics, Kyoto University, Kitashirakawa, Kyoto 606-8502, Japan}
\email{fujioka.tadashi.5t@kyoto-u.ac.jp, tfujioka210@gmail.com}
\date{\today}
\thanks{Supported by JSPS KAKENHI Grant Number 22KJ2099}

\begin{document}

\begin{abstract}
We prove that the extremal stratification of an Alexandrov space introduced by Perelman--Petrunin is a CS stratification in the sense of Siebenmann.
We also show that every space of directions of an Alexandrov space without proper extremal subsets is homeomorphic to a sphere.
In the polyhedral case, the same holds for every iterated space of directions.
\end{abstract}

\maketitle

\section{Introduction}

\subsection{Main result}

The main purpose of this paper is to prove the following.

\begin{thm}\label{thm:main}
The extremal stratification of an Alexandrov space introduced by Perelman--Petrunin is a CS stratification in the sense of Siebenmann.
\end{thm}

An \textit{Alexandrov space} is a metric space with a lower curvature bound introduced by Burago--Gromov--Perelman \cite{BGP}.
They naturally arise as Gromov--Hausdorff limits of Riemannian manifolds with sectional curvature bounded below or quotient spaces of Riemannian manifolds by isometric group actions.
In general, a \textit{stratification} is a well-behaved decomposition of a topological space into topological manifolds of different dimensions.

From a geometrical point of view, the proof of Theorem \ref{thm:main} is merely a combination of known technical results and provides nothing new.
Nevertheless, the author believes it is worth publishing because it connects two concepts in different fields, yet even a statement cannot be found in the literature.
Another motivation comes from the collapsing theory of Riemannian manifolds, where the extremal stratification of limit Alexandrov spaces plays an essential role (\cite{Per:col}, \cite{F:good}, \cite{F:eu}; especially in the sheaf-theoretic conjectures by Alesker \cite{A:conj}).

For this reason, we begin by reviewing what is known about the stratification of Alexandrov spaces.
There are four types of stratification appearing in this paper --- extremal, MCS, CS, and WCS stratifications with the following relationships:
\[
\begin{CD}
\text{extremal} @>{(2)}>> \text{MCS}\;\\
@V{(1)}VV @VV{(4)}V \\
\text{CS} @>>{(3)}> \text{WCS}.
\end{CD}
\]
The precise meaning of each arrow will be explained below.
Here we focus only on the histories of the four notions and defer the definitions to Section \ref{sec:pre}.

The notion of a stratification of an Alexandrov space was first introduced by Perelman \cite{Per:alex}, \cite{Per:mor}, as a result (actually a part) of his Morse theory for distance functions on Alexandrov spaces.
He proved that any Alexandrov space is an \textit{MCS space} (= multiple conical singularity), and in particular that it is stratified into topological manifolds.
Note that the MCS stratification is purely topological.

At the same time as \cite{Per:mor}, Perelman--Petrunin \cite{PP:ext} introduced the notion of an \textit{extremal subset} --- a singular set of an Alexandrov space defined by a purely geometrical condition.
By extending the Morse theory to extremal subsets, they proved that any Alexandrov space is, in a certain sense, uniquely stratified into its own extremal subsets.
Here we refer to it as the \textit{extremal stratification}.
The extremal stratification is finer than the MCS stratification (Arrow (2)): the former reflects not only the topological singularity but also the geometrical singularity of an Alexandrov space.
For example, the extremal stratification of a squire consists of the four vertices, the interiors of the four edges, and the interior of the squire itself, whereas the MCS stratification consists only of the boundary and the interior of the squire.

The above-mentioned Morse theory relies heavily on the work of Siebenmann \cite{Si} on the deformation of homeomorphisms on stratified sets.
In \cite{Si}, Siebenmann introduced two notions of stratified sets: one is a \textit{CS set} (= cone-like stratified) and the other is a \textit{WCS set} (= weakly cone-like stratified).
As the names suggest, CS sets are WCS sets (Arrow (3)).
From our point of view, the main difference of these two is whether the local normal trivialization preserves the stratification or not (see Remark \ref{rem:wcs}).
Still, the fundamental deformation theorem of Siebenmann works in the weaker setting of WCS sets.

In \cite[p.6]{Per:alex}, \cite[p.207]{Per:mor}, Perelman only stated that the MCS stratification of an Alexandrov space is a WCS stratification (Arrow (4)), which was enough for him to develop the Morse theory with the help of Siebenmann's deformation theorem for WCS sets.
However, there was no mention of the CS structure in \cite{Per:alex}, \cite{Per:mor}, even in \cite{PP:ext}, and as already remarked, since then there have been no papers about it (except for a closely related result of Kapovitch \cite[\S9]{K:stab} mentioned below).
This is why we consider this problem, and the answer is, as stated in Theorem \ref{thm:main}, given by the extremal stratification of Perelman--Petrunin (Arrow (1)).

In fact, the proof of Theorem \ref{thm:main} is almost the same as the original proof of the existence of the extremal stratification by Perelman--Petrunin  \cite[3.8, 1.2]{PP:ext}.
The only difference is that we make use of the relative fibration theorem with respect to extremal subsets developed later by Kapovitch \cite[\S9]{K:stab}, instead of the absolute one given in \cite[\S2]{PP:ext}.

\subsection{Further results}

Now we discuss further regularity results for the extremal stratification of Alexandrov spaces, which is a new contribution of this paper.

The extremal stratification theory implies that if an Alexandrov space contains no proper extremal subsets, then it must be a topological manifold.
However, even if an Alexandrov space is a topological manifold, it does not necessarily mean that each link (i.e., the space of directions) is a topological sphere.
See Example \ref{ex:ds} based on the double suspension theorem of Cannon \cite{Ca} and Edwards \cite{E:icm}, \cite{E:susp}.
The infinitesimal characterization of Alexandrov topological manifolds by Wu \cite{W:alex} (cf.\ \cite{W:ed}) only shows that the link is a homology manifold with the homotopy type of a sphere.
Regarding this problem, we prove the following.

\begin{thm}\label{thm:link}
Every space of directions of an Alexandrov space without proper extremal subsets is homeomorphic to a sphere.
\end{thm}

We prove the following more rigid result in Theorem \ref{thm:link'}: the set of points with nonspherical links in an Alexandrov topological manifold forms an extremal subset.
As an application, we obtain a new sphere theorem for the link of an Alexandrov homology manifold with positive curvature (Corollary \ref{cor:link}).
This provides a sufficient, and in some sense optimal, condition for ruling out the double suspension over a nonsimply-connected homology sphere.

Theorem \ref{thm:link} naturally leads us to the conjecture that any Alexandrov space without proper extremal subsets would admit some kind of PL structure.
Let us formulate this conjecture as follows.

\begin{conj}\label{conj:pl}
Suppose an Alexandrov space contains no proper extremal subsets.
Then every iterated space of directions is homeomorphic to a sphere.
\end{conj}

Here an \textit{iterated space of directions} refers to a space of directions of a space of directions of ...\ of a space of directions.
The conclusion of Conjecture \ref{conj:pl} holds for any noncollapsing limit of Riemannian manifolds with sectional curvature bounded below (\cite[1.4]{K:reg}), while the limit space may have proper extremal subsets.
Another sufficient condition for Conjecture \ref{conj:pl} is that every space of directions has \textit{radius} greater than $\pi/2$ (\cite[\S3]{Pet:app}, cf.\ \cite{GP:rad}).
This assumption means that every direction has a $\pi/2$-apart direction, which prevents the existence of proper extremal subsets (\cite[1.4.1]{PP:ext}).
Compare with the corresponding results \cite[1.3, 1.5]{LN:top} in the context of geodesically complete spaces with curvature bounded above.

We shall verify Conjecture \ref{conj:pl} in the polyhedral case.
See Section \ref{sec:poly} for the definition of a polyhedral space.
In this case, the conclusion of Conjecture \ref{conj:pl} is equivalent to the condition that the link around every simplex is homeomorphic to a sphere (Lemma \ref{lem:pl}).

\begin{thm}\label{thm:poly}
For polyhedral Alexandrov spaces, Conjecture \ref{conj:pl} is true.
\end{thm}

\begin{rem}
In the polyhedral case, it would be more natural to expect a PL structure, i.e., the link around every simplex is PL homeomorphic to a sphere.
\end{rem}

\begin{rem}
To the best of the author's knowledge, it is an open problem whether an Alexandrov space admits a triangulation, even when there are no proper extremal subsets.
Note that there exists a topological manifold not admitting any triangulation (see \cite{Ma}).
\end{rem}

\begin{rem}
Although the above theorems only deal with the top stratum of an Alexandrov space, the same results could be extended to each extremal stratum of an Alexandrov space.
However, we will not treat this generalization in this paper, as our primary goal is to prove Conjecture \ref{conj:pl}.
\end{rem}

\step{Organization}
In Section \ref{sec:pre}, we recall background material, including the definitions of the four stratifications mentioned in the introduction.
In Sections \ref{sec:prf}, \ref{sec:link}, and \ref{sec:poly}, we prove Theorems \ref{thm:main}, \ref{thm:link}, and \ref{thm:poly}, respectively.
Although the proof of Theorem \ref{thm:main} is based on several technical results, all we have to do is simply combine them.
For this reason, Sections \ref{sec:pre} and \ref{sec:prf} are rather expository.
On the other hand, Section \ref{sec:link} is technical and demands more familiarity with Alexandrov geometry.
Section \ref{sec:poly}, dealing with polyhedral spaces, requires little knowledge of Alexandrov geometry.
In Appendix \ref{sec:app}, we give an example of a primitive extremal subset of codimension $2$ that is not an Alexandrov space with respect to the induced intrinsic metric.
As with the main theorem, no such examples can be found in the literature.

\step{Acknowledgments}
I would like to thank Semyon Alesker for his original question regarding Theorem \ref{thm:main} that motivated me to write this paper.
I am also grateful to John Harvey, Mohammad Alattar, and Shijie Gu for their valuable comments on an earlier version of this paper.

\section{Preliminaries}\label{sec:pre}

Here we begin with the definition of a CS set and then recall several notions and results from Alexandrov geometry that will be needed to prove Theorem \ref{thm:main}.
In fact, once those results are assumed, the proof requires almost no geometrical knowledge.

\subsection{CS sets}\label{sec:cs}

First we recall the definition of a CS set introduced by Siebenmann \cite[\S1]{Si}.
For simplicity, we restrict our attention to the finite-dimensional case.

A metrizable space $X$ is a \textit{stratified set} if there exists a filtration by closed subsets
\[\emptyset=X^{-1}\subset X^0\subset\dots\subset X^n=X\]
such that the components of $X_k:=X^k\setminus X^{k-1}$ are open in $X_k$.
We call $X^k$ and $X_k$ the \textit{$k$-dimensional skeleta} and \textit{strata}, respectively.
In addition, if the components of $X_k$ are $k$-dimensional topological manifolds (without boundary), then we call $X$ a \textit{TOP stratified set}.

An \textit{isomorphism} between (not necessarily TOP) stratified sets $X$ and $Y$ is a homeomorphism between $X$ and $Y$ preserving their strata of the same dimension.
This will be denoted by $X\cong Y$.

For stratified sets $X$ and $Y$, the direct product $X\times Y$ has a natural stratification.
The open cone
\[C(X):=X\times[0,1)/X\times\{0\}\]
has a natural stratification as well, which is induced from the product structure of $X\times(0,1)$ together with the zero-dimensional stratum consisting of the vertex of the cone.
An open subset $U$ of $X$ also has a natural stratification by restriction.

\begin{dfn}\label{dfn:cs}
A TOP stratified set $X$ is called a \textit{CS set} if, for any $x\in X_k$, there exist an open neighborhood $U$ of $x$ in $X$ and a compact stratified set $L$ together with an isomorphism
\[h:U\xrightarrow\cong\ B^k\times C(L),\]
where $B^k$ denotes an open Euclidean ball of dimension $k$.
Note that both sides have natural stratifications defined above, where the stratification of $B^k$ is trivial.
We call $L$ a \textit{(normal) link} and $h$ a  \textit{local normal trivialization} at $x$.
\end{dfn}

\begin{rem}
In general, $L$ need not be either a CS or TOP stratified set.
Moreover, its homeomorphism type is not necessarily unique (see Example \ref{ex:ds}).
\end{rem}

\begin{rem}\label{rem:wcs}
There is a weaker notion of a \textit{WCS set} introduced by Siebenmann \cite[\S5]{Si}.
There are two relaxations of the above conditions:
\begin{enumerate}
\item $C(L)$ may be a \textit{mock cone}, not necessarily a real cone;
\item $h$ may be a homeomorphism preserving only the $k$-dimensional stratum, not necessarily an isomorphism preserving the higher-dimensional strata.
\end{enumerate}
Any Alexandrov space is known to be a WCS set, but not known to be a CS set in the literature.
As we will see in Section \ref{sec:mcs}, the problem lies in (2) rather than (1) (i.e., $C(L)$ is a real cone but $h$ might not be an isomorphism).
In particular, we do not need the definition of the mock cone here.
\end{rem}

\subsection{Alexandrov spaces}

Next we recall the definition of an Alexandrov space introduced by Burago--Gromov--Perelman \cite{BGP} (though we will not use it directly).
See also \cite{BBI}, \cite{AKP:found}, \cite{Pl} for the basic theory of Alexandrov spaces.

Let $X$ be a metric space.
A \textit{shortest path} in $X$ is an isometric embedding of a closed interval into $X$.
We say that $X$ is a \textit{geodesic space} if every two points can be joined by a shortest path.

Fix $\kappa\in\mathbb R$.
The \textit{model plane} of curvature $\kappa$ refers to a sphere, Euclidean plane, or a hyperbolic plane of constant curvature $\kappa$.
For a geodesic triangle $\triangle pqr$ in a metric space, the \textit{comparison triangle} $\tilde\triangle pqr$ means a geodesic triangle in the model plane with the same side-lengths (if it exists and is unique up to isometry when $\kappa>0$).

\begin{dfn}
A complete geodesic space $X$ is called an \textit{Alexandrov space} with curvature $\ge\kappa$ if for any geodesic triangle $\triangle pqr$ in $X$ and the comparison triangle $\tilde\triangle pqr$ in the model plane of curvature $\kappa$, the natural correspondence
\[\triangle pqr\to\tilde\triangle pqr\]
is distance-nondecreasing.
\end{dfn}

In particular, this enables us to define the \textit{angle} $\angle pqr$ for every pair of shortest paths $qp$ and $qr$ (see the above references).
The notion of angle plays a fundamental role in the geometry of Alexandrov spaces.

\begin{ex}
The following are Alexandrov spaces with curvature $\ge\kappa$.
\begin{itemize}
\item A Riemannian manifold $M$ with sectional curvature $\ge\kappa$.
\item A convex subset and a convex hypersurface in $M$ (\cite{AKP:opt}).
\item A quotient space $M/G$, where $G$ is a compact group acting on $M$ by isometries.
\item A Gromov--Hausdorff limit of Riemannian manifolds $M_j$ with sectional curvature $\ge\kappa$.
\end{itemize}
\end{ex}

In this paper, we always assume the finite-dimensionality of Alexandrov spaces.
Here the \textit{dimension} of an Alexandrov space means its Hausdorff dimension, which coincides with the topological dimension.

\subsection{MCS spaces}\label{sec:mcs}

Here we recall another notion of stratified space introduced by Perelman \cite{Per:alex}, \cite{Per:mor}.

\begin{dfn}\label{dfn:mcs}
A metrizable space $X$ is called an \textit{MCS space} of dimension $n$ if, for any $x\in X$, there exist a neighborhood $U$ of $x$ in $X$ and a compact MCS space $L$ of dimension $n-1$ together with a pointed homeomorphism
\[(U,x)\to(C(L),o)\]
where $o$ is the vertex of the cone $C(L)$.
Here we assume that the empty set is the unique compact $(-1)$-dimensional MCS space, the cone over which is a point.
\end{dfn}

\begin{rem}
The conical neighborhood is unique up to pointed homeomorphism, as shown by Kwun \cite{Kw} (though the link may be different).
\end{rem}

Every MCS space $X$ has a natural stratification.
The $k$-dimensional stratum $X_k$ consists of those points whose conical neighborhood topologically splits into $B^k\times C(L)$ but not $B^{k+1}\times C(L')$, where $L$ and $L'$ are compact MCS spaces.
It is easy to see that the MCS stratification is a WCS stratification mentioned in Remark \ref{rem:wcs}.
More precisely, if $x\in X_k$, then there exist an open neighborhood $U$ of $x$ in $X$ and a compact MCS space $L$ such that there is a homeomorphism
\[h:U\to\ B^k\times C(L)\]
that sends the bottom stratum $U\cap X^k$ to $B^k\times\{o\}$, where $o$ is the vertex of $C(L)$.
The reason we cannot say this is a CS stratification is that it is unclear whether $h$ preserves the higher-dimensional strata, as noted in Remark \ref{rem:wcs}.

Perelman proved the following theorem.

\begin{thm}[\cite{Per:alex}, \cite{Per:mor}]
Every $n$-dimensional Alexandrov space is an MCS space of dimension $n$.
In particular, it is a WCS set.
\end{thm}

\begin{rem}
The proof actually shows that every Alexandrov space is a \textit{nonbranching} MCS space.
This is defined in the same inductive way as in Definition \ref{dfn:mcs}, but allowing only $1$-dimensional manifolds with boundary to be $1$-dimensional nonbranching MCS spaces.
\end{rem}

In order to prove the above theorem, Perelman developed the Morse theory for distance maps on Alexandrov spaces.
We will review it in the next section (see also \cite[\S2]{Per:icm}, \cite{K:stab}).

\subsection{Canonical neighborhoods}

The following theorem is proved by reverse induction on $k$, using the deformation theorem of Siebenmann for WCS sets \cite[5.4, 6.9, 6.14]{Si}.
Let $X$ be an Alexandrov space of dimension $n$.

\begin{thm}[{\cite[1.3]{Per:mor}}]\label{thm:cn}
For any $p\in X$, there exist an open neighborhood $K$ of $p$ in $X$ and a compact MCS space $L$ such that there is a homeomorphism
\[h:K\to B^k\times C(L),\]
where $0\le k\le n$.
More precisely, there exists a map $(f,g):K\to B^k\times[0,1)$ with the following properties:
\begin{enumerate}
\item $g(p)=0$ (possibly $g\equiv 0$ when $k=n$);
\item $f$ is regular on $K$ and $(f,g)$ is regular on $K\setminus g^{-1}(0)$;
\item $f$ restricted to $g^{-1}(0)$ is a homeomorphism to $B^k$ and $(f,g)$ restricted to $K\setminus g^{-1}(0)$ is a trivial fiber bundle over $B^k\times(0,1)$ with fiber $L$.
\end{enumerate}
In particular, the homeomorphism $h$ is constructed from (3): the first $B^k$-coordinate is given by $f$ and the projection to the second factor of $C(L)=L\times[0,1)/L\times\{0\}$ is given by $g$.
\begin{figure}[h]
\centering
\begin{tikzpicture}
\coordinate[label=below:$p$](p)at(0,0);

\draw(-3,-0)to(3,0);
\draw(p)to(1,2);
\draw(p)to(-1,2);
\draw[dashed](0,1.5)circle[x radius=0.75,y radius=0.25]node{$L$};

\draw[->](-3,-0.5)to(3,-0.5)node[right]{$f$};
\draw[->](3.5,0)to(3.5,2)node[above]{$g$};

\fill(p)circle(1.5pt);
\end{tikzpicture}
\end{figure}
\end{thm}

We will not state the definition of a regular map here (see \cite{Per:mor}, \cite{K:stab}).
It would be more helpful to give a basic example: for $k+1$ (!)\ points $a_i\in X$ ($0\le i\le k$), the distance map
\[f=(d(a_1,\cdot),\dots,d(a_k,\cdot))\]
is \textit{regular} at $p$ if the angles $\angle a_ipa_j$ are all greater than $\pi/2$, where $0\le i\neq j\le k$.
This condition should be regarded as linear independence of the gradients of the coordinate functions of $f$.
The precise definition is more complicated to include the function $g$ in Theorem \ref{thm:cn}.

\begin{rem}\label{rem:fib}
Any regular map is open.
Moreover, the second half of Theorem \ref{thm:cn}(3) is a consequence of the regularity of $(f,g)$ and the \textit{fibration theorem} proved simultaneously in \cite{Per:mor}, asserting that any proper regular map is a locally trivial fibration.
\end{rem}

\begin{rem}\label{rem:cn}
Given a regular map $f_0:X\to\mathbb R^{k_0}$ at $p$, the map $f$ of Theorem \ref{thm:cn} can be chosen to \textit{respect} $f_0$, that is, the first $k_0$-coordinates of $f$ coincide with $f_0$, where $k_0\le k$.
\end{rem}

Theorem \ref{thm:cn} allows us to introduce the following notion, which plays a key role in the proof of Theorem \ref{thm:main}.

\begin{dfn}\label{dfn:cn}
The subset $K$ of Theorem \ref{thm:cn} is called a \textit{canonical neighborhood} of $p$ associated with the map $(f,g)$.
We also denote it by $K(f,g)$ to indicate $(f,g)$.
If $f$ respects $f_0$ in the above sense, we say that $K$ \textit{respects} $f_0$.
\end{dfn}

\subsection{Extremal subsets}

We now recall the definition of an extremal subset introduced by Perelman--Petrunin \cite{PP:ext}.
See also \cite[\S3]{Per:icm}, \cite[\S4]{Pet:sc}, \cite[\S9.7]{Pl}, \cite{F:reg}, \cite{F:uni} for the basic theory.

Roughly speaking, an extremal subset is a closed subset of an Alexandrov space such that the set of normal directions at each point has diameter $\le\pi/2$.
The precise definition is as follows (but we will not directly use it in this paper).

\begin{dfn}\label{dfn:ex}
A closed subset $E$ of an Alexandrov space $X$ is called \textit{extremal} if the following holds:
for any $q\in X\setminus E$, if the distance function $d(q,\cdot)$ restricted to $E$ attains a local minimum at $p\in E$, then $p$ is a critical point of $d(q,\cdot)$, i.e.,
\[\angle qpx\le \pi/2\]
for every $x\in X$.
By convention, $\emptyset$ and $X$ are extremal subsets.
\end{dfn}

Note that the above definition of a critical point is compatible with the regularity of a distance function in the previous section.

\begin{ex}
The following are extremal subsets.
\begin{itemize}
\item A point at which all angles $\le\pi/2$.
This called an \textit{extremal point}.
\item The closures of the MCS strata of an Alexandrov space (see Section \ref{sec:mcs}); especially the \textit{boundary}, i.e., the closure of the codimension $1$ stratum.
\item If a compact group $G$ acts isometrically on an Alexandrov space $X$, then $X/G$ is an Alexandrov space.
The projection of  the fixed point set of any closed subgroup $H$ of $G$ is an extremal subset of $X/G$.
\end{itemize}
\end{ex}

We will recall how to define the extremal stratification of an Alexandrov space.
The following basic properties were proved in \cite{PP:ext}.
\begin{itemize}
\item The union and intersection of two extremal subsets are extremal.
Moreover, the closure of the difference of the two is extremal.
\item The collection of extremal subsets is locally finite.
\end{itemize}
These facts lead to the following definition.

\begin{dfn}
An extremal subset $E$ is called \textit{primitive} if it cannot be represented as a union of two proper extremal subsets, i.e., if
\[E=E_1\cup E_2,\]
where $E_1$ and $E_2$ are extremal, then $E=E_1$ or $E=E_2$.
For a primitive extremal subset $E$, its \textit{main part} $\mathring E$ is $E$ with all its proper extremal subsets removed.
\end{dfn}

\begin{rem}
In other words, a primitive $E$ is a minimal extremal subset containing some point $x\in X$. The main part of $E$ is the set of those $x$ for which $E$ is the minimal extremal subset containing it.
\end{rem}

Clearly, the main parts of all primitive extremal subsets form a disjoint covering of an Alexandrov space.
Furthermore, Perelman--Petrunin \cite[\S2]{PP:ext} developed the Morse theory on extremal subsets by extending Theorem \ref{thm:cn} to the restriction (!)\ of a regular map to extremal subsets.
In particular, they proved in \cite[3.8, 1.2]{PP:ext} the following:
\begin{itemize}
\item Any primitive extremal subset $E$ is an MCS space.
\item The closures of the MCS strata $E_k$ of $E$ are extremal subsets.
\end{itemize}
(Technically Perelman--Petrunin introduced a variant of MCS space and showed that every extremal subset is this variant, but for primitive ones, we do not need this notion.)
The above two facts imply that the main part of a primitive extremal subset is a connected topological manifold and is contained in some MCS stratum of the ambient Alexandrov space.
Let us summarize this as follows.

\begin{dfn}
The \textit{extremal stratification} of an Alexandrov space is defined in such a way that the $k$-dimensional stratum consists of the $k$-dimensional main parts of primitive extremal subsets.
\end{dfn}

\begin{thm}[{\cite[3.8]{PP:ext}}]
The extremal stratification of an Alexandrov space is a TOP stratification that is a refinement of the MCS stratification.
\end{thm}

\begin{rem}\label{rem:dim}
The dimension of the main part of a primitive extremal subset $E$ coincides with the Hausdorff and topological dimensions of $E$.
We simply call it the \textit{dimension} of $E$ (\cite{F:reg}).
It also coincides with the maximal integer $k$ such that there exists a regular map to $\mathbb R^k$ at some point of $E$ (\cite[\S2]{PP:ext}, \cite[p.133]{K:stab}).
\end{rem}

Next we discuss the relationship between extremal subsets and canonical neighborhoods.
The following characterization of extremal subsets is due to Perelman \cite[2.3]{Per:col} (see also \cite[5.1]{F:good}).
Let $X$ be an Alexandrov space.

\begin{lem}\label{lem:ext}
Suppose a subset $S$ of $X$ satisfies the following property:
\begin{enumerate}
\item[($\ast$)] if  $S$ intersects a canonical neighborhood $K(f,g)$, then $S$ contains $g^{-1}(0)$.
\end{enumerate}
Then the closure of $S$ is an extremal subset.
\end{lem}

\begin{rem}\label{rem:con}
The converse is true, that is, if an extremal subset $E$ intersects $K(f,g)$, then $E$ contains $g^{-1}(0)$.
This follows from the openness of a regular map restricted to an extremal subset, \cite[\S2]{PP:ext}.
\end{rem}

The next relative version of Theorem \ref{thm:cn} with respect to extremal subsets was proved by Kapovitch \cite[\S9]{K:stab}.

\begin{comp}[to Theorem \ref{thm:cn}]\label{comp:cn}
The homeomorphism $h$ of Theorem \ref{thm:cn} can be chosen to respect all extremal subsets, i.e., for any extremal subset $E$ intersecting $K$, the restriction of $h$ gives a homeomorphism to the subcone:
\[K\cap E\to B^k\times C(L\cap E).\]
\end{comp}

As in Remark \ref{rem:fib}, this also follows from the \textit{relative fibration theorem} proved simultaneously in \cite[\S9]{K:stab}, asserting that any proper regular map is a locally trivial fibration respecting extremal subsets.
Indeed, applying it to the regular map $(f,g)$ on $K\setminus g^{-1}(0)$ and gluing the homeomorphism $f$ on $g^{-1}(0)\subset E$ yield the desired homeomorphism $h$.
The proof relies on the relative version of the deformation theorem of Siebenmann \cite[5.10, 6.10, 6.16]{Si}.

\begin{rem}\label{rem:comp}
In the above statement, it is possible that $L\cap E=\emptyset$; in this case $K\cap E=g^{-1}(0)$ corresponds to $B^k\times\{o\}$, where $o$ is the vertex of the cone $C(L)$.
\end{rem}

\begin{rem}
Kapovitch only discussed the case of one extremal subset, but in view of the relative deformation theorem of Siebenmann cited above, it is easy to generalize it to the multiple case (as observed in \cite[4.3]{HS}).
\end{rem}

The key difference between the absolute version of the Morse theory on extremal subsets by Perelman--Petrunin \cite[\S2]{PP:ext} and the relative version by Kapovitch \cite[\S9]{K:stab} is whether it includes information on the normal directions or not.
The latter is essential for proving Theorem \ref{thm:main}.

\subsection{Double suspension}
We conclude this section with a well-known example from geometric topology.
From our point of view, this gives an example for which the extremal stratification does not coincide with the MCS stratification, due to the existence of nonspherical links.

For a metric space $\Sigma$ with diameter $\le\pi$, we denote by $S(\Sigma)$ the \textit{spherical suspension} over $\Sigma$, that is,
\[S(\Sigma)=\Sigma\times[0,\pi]/\sim,\]
where we identify each $\Sigma\times\{0\}$ and $\Sigma\times\{\pi\}$ with a point, and equip $S(\Sigma)$ with a spherical metric regarding $\Sigma\times\{\pi/2\}$ as an equator (see \cite[3.6.3]{BBI} for more details).
We also denote the \textit{$k$-fold spherical suspension} over $\Sigma$ by
\[S^k(\Sigma)=\underbrace{S(S(\cdots S}_k(\Sigma)\cdots)).\]

\begin{ex}\label{ex:ds}
Let $\Sigma$ be the Poincar\'e homology sphere of constant curvature $1$.
Then we have
\begin{itemize}
\item $S(\Sigma)$ is not a topological manifold;
\item $S^2(\Sigma)$ is a topological sphere by Cannon \cite{Ca} and Edwards \cite{E:icm}, \cite{E:susp}.
\end{itemize}
In particular, the MCS stratification of $S^2(\Sigma)$ is trivial.
On the other hand, since $\diam\Sigma\le\pi/2$ (cf.\ \cite{GS}), the extremal stratification consists of the subsuspension $S(\{\xi,\eta\})$ and its complement, where $\xi,\eta$ are the poles of $S(\Sigma)$.
Notice that the extremal subset $S(\{\xi,\eta\})$ coincides with the set of points with nonspherical links.
\end{ex}

\begin{rem}
Theorem \ref{thm:link} tells us that the above extremal subset must exist.
See Theorem \ref{thm:link'} and Corollary \ref{cor:link} for further discussion.
\end{rem}

\section{Proof of Theorem \ref{thm:main}}\label{sec:prf}

Now we prove Theorem \ref{thm:main}.
Almost all the background knowledge needed for the proof is covered in the previous section.

\begin{proof}[Proof of Theorem \ref{thm:main}]
Let $X$ be an Alexandrov space and $E$ a primitive extremal subset of dimension $m$.
We will show that there exists a compact stratified set $L$ (depending only on $E$) satisfying the following:
for any point $x$ in the main part $\mathring E$ of $E$, there is a neighborhood $U$ of $x$ in $X$ that admits an isomorphism as stratified sets
\begin{equation}\label{eq:prf}
U\cong B^m\times C(L).
\end{equation}
Here the stratifications of the left- and right-hand sides are the natural ones induced from those of $X$ and $L$, respectively.

\step{Step 1}
We first prove \eqref{eq:prf} for any regular point of $E$.
A point $p\in E$ is called \textit{regular} if the tangent cone $T_pE$ of $E$ at $p$ is isometric to $\mathbb R^m$, where $m=\dim E$.
Here the tangent cone means the Gromov--Hausdorff blowup; see \cite{F:reg} for the details.
Note that the set of regular points of $E$ is dense in $E$ and is contained in $\mathring E$ (\cite[1.1]{F:reg}, Remark \ref{rem:dim}).

Let $p\in E$ be a regular point.
Choose $(m+1)$ directions $\xi_i$ ($0\le i\le m$) in the tangent cone $T_pE=\mathbb R^m$ that make obtuse angles at the vertex and let $a_i\in X$ be points near $p$ in the directions $\xi_i$.
Then the distance map
\[f=(d(a_1,\cdot),\dots,d(a_m,\cdot))\]
is regular at $p$.
Since $\dim E=m$, there is no function $f_1$ such that $(f,f_1)$ is regular at $p$ (Remark \ref{rem:dim}).
By Theorem \ref{thm:cn} and Remark \ref{rem:cn}, one can construct a function $g$ and a canonical neighborhood $K=K(f,g)$ of $p$ together with a homeomorphism
\begin{equation}\label{eq:reg}
h:K\to B^m\times C(L_p),
\end{equation}
where $L_p$ is a fiber of $(f,g)$ in $K\setminus g^{-1}(0)$.

It remains to show that $h$ preserves the stratifications of both sides of \eqref{eq:reg}.
Note that not only $K$, but also $L_p$ has a natural stratification as a subset of $X$, by restricting the extremal stratification of $X$.
Since $L_p$ is not open, we have to show that this is a stratified set in the sense of Section \ref{sec:cs}, i.e., the components of the strata are open.
However, we defer it to the next step and first observe that $h$ preserves these stratifications.

By Complement \ref{comp:cn}, $h$ can be chosen to respect all extremal subsets.
Note that the right-hand side of \eqref{eq:reg} may have one extra stratum $B^m\times\{o\}$ arising from the cone structure of $C(L_p)$.
We have to show that there is a bottom stratum in the left-hand side that corresponds to $B^m\times\{o\}$.
This is nothing but $E$ itself by the same reasoning $\dim E=m$ as above (Theorem \ref{thm:cn}(2), Remark \ref{rem:dim}).

\step{Step 2}
As mentioned above, we prove that the components of the strata of $L_p$ are open.
It suffices to show that the strata are locally path-connected.
Let $E_1$ be a primitive extremal subset of $X$ intersecting $L_p$ and let $x\in L_p\cap\mathring E_1$.
By Theorem \ref{thm:cn} and Remark \ref{rem:cn}, there exists a canonical neighborhood $K_1$ of $x$ respecting $(f,g)$ (see Definition \ref{dfn:cn}).
By Complement \ref{comp:cn}, $K_1\cap L_p\cap\mathring E_1$ is homeomorphic to a product of a ball and a cone, which is locally path-connected at $x$.

\step{Step 3}
Finally we prove \eqref{eq:prf} for any point of the main part $\mathring E$.
Fix a regular point $p\in E$ and set $L:=L_p$.
Define a subset $F$ of $E$ by
\[F:=\left\{x\in E\mid\text{no open neighborhood in $X$ is isomorphic to $B^m\times C(L)$}\right\},\]
where the isomorphism is as stratified sets.
Note that $F$ is closed since the bottom stratum of $B^m\times C(L)$ has dimension $m$, while $\mathring E$ is the unique $m$-dimensional stratum contained in $E$.

We prove that $F$ is an extremal subset of $X$.
It suffices to verify Condition ($*$) of Lemma \ref{lem:ext} for $F$.
Suppose $x\in E$ has a small neighborhood isomorphic to $B^m\times C(L)$.
For a canonical neighborhood $K(f_1,g_1)$, we show that if $x\in g_1^{-1}(0)$, then any $y\in g_1^{-1}(0)$ has a neighborhood isomorphic to $B^m\times C(L)$ (note that $y\in E$ follows from Remark \ref{rem:con}).
Consider a homeomorphism
\[h_1:K(f_1,g_1)\to B^k\times C(L_1)\]
respecting extremal subsets as in Complement \ref{comp:cn} and a homeomorphism
\[T_1:B^k\times C(L_1)\to B^k\times C(L_1)\]
that sends $h_1(x)=(f_1(x),o)$ to $h_1(y)=(f_1(y),o)$ and fixes the $C(L_1)$-coordinate.
The composition $h_1^{-1}\circ T_1\circ h_1$ shows that $x$ and $y$ have isomorphic neighborhoods.
Thus Condition ($*$) is satisfied and hence $F$ is an extremal subset of $X$.

Since $p\notin F$, $F$ is a proper subset of $E$.
By the definition of the main part, $F$ is contained in $E\setminus\mathring E$.
Therefore, any point of $\mathring E$ has a neighborhood isomorphic to $B^m\times C(L)$.
This complete the proof.
\end{proof}

\begin{rem}
In Step 3, we do not really need Lemma \ref{lem:ext} to prove that $F$ is an extremal subset.
One can just repeat the same argument as in \cite[3.8, 1.2]{PP:ext}, which showed that $\mathring E$ is a topological manifold, by replacing the absolute fibration theorem on extremal subsets \cite[\S2]{PP:ext} with the relative one \cite[\S9]{K:stab}.
The reason we did not take this approach is because our exposition puts emphasis on canonical neighborhoods.
Our argument can be viewed as a fixed space version of \cite[5.6, 5.11]{F:good}, which showed that the homotopy type of a regular fiber in a collapsing Alexandrov space is invariant over a component of the extremal strata of the limit space.
\end{rem}

By the theorem of Handel \cite[2.4]{Ha}, we obtain the following corollary.

\begin{cor}
The intrinsic stratification of an Alexandrov space is CS.
\end{cor}

Here the \textit{intrinsic stratification} (or \textit{minimal stratification}) is defined as follows.
The $k$-dimensional stratum consists of those points whose conical neighborhood topologically splits into $B^k\times C(L)$ but not $B^{k+1}\times C(L')$, where $L$ and $L'$ are compact topological spaces.
Compare with the MCS stratification in Section \ref{sec:mcs}, where $L$ and $L'$ are MCS spaces.
The following question still remains.

\begin{ques}
Is the MCS stratification of an Alexandrov space different from the intrinsic stratification?
Even if so, is it CS?
\end{ques}

Another question remains regarding the link (cf.\ Theorem \ref{thm:link}, Conjecture \ref{conj:pl}).

\begin{ques}
Is the stratification of the link $L$ in the above proof TOP, or more strongly CS?
If it is CS, what about iterated links? (i.e., a link of a link of ...)
\end{ques}

\begin{rem}
Note that $L$ is an MCS space (Theorem \ref{thm:cn}).
Moreover, in view of the stability theorem \cite[4.3]{Per:alex}, $L$ is homeomorphic to the set $\Sigma$ of directions at $p$ normal to the tangent cone $T_pE$.
Since $p$ is a regular point, $\Sigma$ has curvature $\ge 1$.
Thus $\Sigma$ admits the extremal stratification, which gives $L$ a CS structure.
However, this CS stratification is generally finer than the above stratification of $L$ induced from the extremal stratification of the ambient space $X$.
See Example \ref{ex:ab}.
\end{rem}

\section{Proof of Theorem \ref{thm:link}}\label{sec:link}

Next we prove Theorem \ref{thm:link}.
Here we assume more familiarity with Alexandrov geometry and the properties of extremal subsets.

We first collect the fundamental properties of Alexandrov homology manifolds (with integer coefficients) that will be used throughout this section.
By a \textit{homology $n$-manifold} we mean a topological space whose local homology group is isomorphic to that of $\mathbb R^n$.

Let $X$ be an Alexandrov homology manifold and $p\in X$.
\begin{enumerate}
\item The space of directions $\Sigma_p$ at $p$ is an Alexandrov homology manifold with the homology of a sphere.
\item The set of nonmanifold points in $X$ is discrete; in particular, if $p$ is a nonmanifold point, then it is an extremal point.
\item The point $p$ is a manifold point if and only if $\Sigma_p$ is simply-connected, or equivalently, homotopy equivalent to a sphere (by (1) and the Whitehead theorem).
\item Every Alexandrov homology manifold of dimension $\le 3$ is a topological manifold.
\end{enumerate}

For the proofs, we refer to \cite{W:alex} (cf.\ \cite{W:ed}).
More precisely, Properties (1)--(4) can be found in \cite[3.1, 3.5, 1.2, 3.3]{W:alex}, respectively.
Most of these properties also follow from Theorem \ref{thm:main} and \cite{He}, \cite{Q:obs} (the author thanks Mohammad Alattar for this information).
See also \cite[7.2]{Q:prob} and the subsequent paragraph.
We remark that the corresponding results for geodesically complete spaces with curvature bounded above can be found in \cite{LN:top}.
The proofs there, which do not rely on strong structural results as in the Alexandrov case, also apply to Alexandrov spaces with curvature bounded below (if correctly translated).

We now prove Theorem \ref{thm:link}.
We show the following more rigid result.

\begin{thm}\label{thm:link'}
Let $X$ be an Alexandrov topological manifold.
Let $E$ be the set of points in $X$ at which the space of directions is not homeomorphic to a sphere.
If $E$ is nonempty, then it is a $1$-dimensional extremal subset.
\end{thm}

Since any Alexandrov space without proper extremal subsets is a topological manifold, Theorem \ref{thm:link'} immediately implies Theorem \ref{thm:link}.

In order to prove Theorem \ref{thm:link'}, we recall the infinitesimal characterization of extremal subsets proved in \cite[1.4]{PP:ext} (cf.\ \cite[\S4.1 Property (2)]{Pet:sc}).
Let $X$ be an Alexandrov space and $E$ a subset of $X$ (it need not be extremal or even closed).
For any point $p$ in the closure of $E$, we denote by $\Sigma_pE$ the \textit{space of directions} of $E$ at $p$, that is, the subset of $\Sigma_p$ consisting of limits of the directions of shortest paths from $p$ to $p_i\in E\setminus\{p\}$, where $p_i$ converges to $p$.
Note that $\Sigma_pE$ is always closed, while $E$ may be not.

\begin{lem}[{\cite[1.4]{PP:ext}}]\label{lem:inf}
Let $X$ be an Alexandrov space and $E$ a closed subset.
Then $E$ is extremal in $X$ if and only if $\Sigma_pE$ is extremal in $\Sigma_p$ for any $p\in E$ and satisfies the following additional conditions:
\begin{enumerate}
\item if $\Sigma_pE=\emptyset$, then we have $\diam\Sigma_p\le\pi/2$;
\item if $\Sigma_pE=\{\xi\}$, then we have $\Sigma_p=\bar B(\xi,\pi/2)$.
\end{enumerate}
Here $\diam$ denotes the diameter and $\bar B(\cdot,\cdot)$ denotes a closed metric ball.
\end{lem}

\begin{proof}[Proof of Theorem \ref{thm:link'}]
By Properties (1), (3) and (4), in dimension $\le4$, every space of directions of an Alexandrov topological manifold is automatically homeomorphic to a sphere (via the Poincar\'e conjecture).
Therefore we may restrict our attention to dimension $\ge 5$.

Let $X$ be an Alexandrov topological manifold and $p\in X$.
By Properties (1) and (3), the space of directions $\Sigma_p$ is a homology manifold with the homotopy type of a sphere.
In particular, $\Sigma_p$ is homeomorphic to a sphere if and only if it is a topological manifold (by the generalized Poincar\'e conjecture).
By Properties (1) and (3) again, $\Sigma_p$ is a topological manifold if and only if all its space of directions are simply-connected.
In summary,
\begin{align*}
&\text{$\Sigma_p$ is a topological sphere}\\
\iff&\text{$\Sigma_p$ is a topological manifold}\\
\iff&\text{every space of directions of $\Sigma_p$ is simply-connected}.
\end{align*}
Furthermore, if $\Sigma_p$ has nonmanifold points, then they are isolated extremal points (Property (2)).
These facts will be used frequently below.

To prove that $E$ is an extremal subset, it is sufficient to show that
\begin{itemize}
\item $E$ is closed;
\item $\Sigma_pE$ is an extremal subset of $\Sigma_p$ for any $p\in E$;
\item $\Sigma_pE$ satisfies the additional conditions as in Lemma \ref{lem:inf}.
\end{itemize}
In particular, for the second claim, we will prove in Claim \ref{clm:link} that $\Sigma_pE$ coincides with the set of nonmanifold points in $\Sigma_p$, which is a discrete set of extremal points (Property (2)).

\step{Step 1}
We first show that if $\Sigma_p$ contains only one nonmanifold point $\xi$, then $\Sigma_p$ is contained in the closed $\pi/2$-neighborhood of $\xi$.
This is an additional condition required to apply Lemma \ref{lem:inf} (as we will see below, we do not have to consider the other case $\Sigma_pE=\emptyset$).
Actually we prove that $\diam\Sigma_p\le\pi/2$.

Suppose $\diam\Sigma_p>\pi/2$.
Then the suspension theorem of Perelman \cite[4.5]{Per:alex} implies that $\Sigma_p$ is homeomorphic to a suspension $S(\Sigma)$, where $\Sigma$ is a space of curvature $\ge 1$.
Moreover, its proof shows that the distance function to each pole of $S(\Sigma)$ is regular except at the poles.
In particular, any point other than the poles is not an extremal point, and thus a manifold point (Property (2)).
Therefore the nonmanifold point $\xi$ must be one of the poles of $S(\Sigma)$.
Let $\eta$ be the other pole.
By the suspension structure and the stability theorem \cite[4.3]{Per:alex}, we see that $\Sigma_\eta$ is homotopy equivalent to $\Sigma_\xi$ (more precisely, they are both homotopy equivalent to $\Sigma$).
Since $\xi$ is a nonmanifold point, this implies that $\eta$ is also a nonmanifold point (Property (3)).
This contradicts the assumption that $\xi$ is the only nonmanifold point of $\Sigma_p$.

\step{Step 2}
We next prove the following key result.

\begin{clm}\label{clm:link}
Let $p\in X$ and $\xi\in\Sigma_p$.
Then $\xi$ is a nonmanifold point if and only if there exists a sequence $p_i\in E\setminus\{p\}$ converging to $p$ such that the directions of shortest paths from $p$ to $p_i$ converges to $\xi$.
In other words,
\[\Sigma_pE=N(\Sigma_p),\]
where $N(\Sigma_p)$ denotes the set of nonmanifold points in $\Sigma_p$.
\end{clm}

The idea of the proof is as follows.
Suppose a neighborhood of $p$ is isometric to a metric cone over $\Sigma_p$.
Let $x\in X$ be a point in the direction $\xi\in\Sigma_p$.
Then $\Sigma_x$ is isometric to the spherical suspension $S(\Sigma_\xi)$, where $\Sigma_\xi$ is the space of directions of $\Sigma_p$ at $\xi$.
As in Step 1, $\Sigma_x$ is homeomorphic to a sphere if and only if $\Sigma_\xi$ is simply-connected.
The former is equivalent to $\xi\in\Sigma_pE$, and the latter is equivalent to $\xi\in N(\Sigma_p)$ (Property (3)).
We adapt this argument to the general setting of Alexandrov spaces, with the help of technical results of Perelman \cite{Per:alex}.

\begin{proof}[Proof of Claim \ref{clm:link}]
For $a,b\in X$, we denote by $b'_a\in\Sigma_a$ the direction of a shortest path from $a$ to $b$.
We also denote by $B(p,r)$ and $\partial B(p,r)$ the open $r$-ball and $r$-sphere around $p$, respectively.
Suppose $r$ is small enough and let $x\in\partial B(p,r)$.
Then the point $p$ is a \textit{$1$-strainer} at $x$, that is, there exists a direction at $x$ making an angle almost $\pi$ with any $p'_x$.
Next choose $\varepsilon>0$ small enough compared to $r$.
Then the pair of points $(p,x)$ is a \textit{$2$-strainer} at any $y\in\partial B(p,r)\cap\partial B(x,\varepsilon)$, i.e., $p$ and $x$ are $1$-strainers at $y$ and $p'_y$ and $x'_y$ are almost orthogonal.
In particular, the distance map $(d(p,\cdot),d(x,\cdot))$ is \textit{noncritical} at $y$ in the sense of \cite[3.1]{Per:alex}.
See also \cite{BBI}, \cite{BGP} for the precise definition and properties of strainers.

\begin{figure}[h]
\centering
\begin{tikzpicture}
\coordinate[label=below:$p$](p)at(0,0);
\coordinate[label=right:$x$](x)at(45:2);

\draw[name path=r](p)circle[radius=2];
\draw[name path=e](x)circle[radius=0.5];

\path[name intersections={of=r and e}];
\coordinate[label=right:\textcolor{red}{$\Sigma$}](s)at(intersection-2);
\coordinate(s')at(intersection-1);

\draw[red,thick,densely dotted](s)to[out=150,in=-60](s');
\draw[dashed](p)circle[x radius =2,y radius=0.5];

\node[below right]at(-45:2){$\partial B(p,r)$};
\node[right]at(45:2.5){$\partial B(x,\varepsilon)$};

\fill(p)circle(1.5pt);
\fill(x)circle(1.5pt);
\end{tikzpicture}
\end{figure}

We will show that $\Sigma_x$ is homeomorphic to the topological suspension over
\[\Sigma:=\partial B(p,r)\cap\partial B(x,\varepsilon).\]
For notational simplicity, suppose that $p'_x$ is unique (the general case is the same by regarding $p'_x$ as the set of all directions of shortest paths from $x$ to $p$).
Let $\eta\in\Sigma_x$ be the unique direction farthest from $p'_x$ and $\zeta\in\Sigma_x$ be the unique direction farthest from $\eta$ (note that $\eta$ and $\zeta$ are at distance almost $\pi$).
Then $\zeta$ is $\delta(r)$-close to $p'_x$, where $\delta(r)\to0$ as $r\to 0$.
By the suspension theorem \cite[4.5]{Per:alex} (see also \cite[2.8]{GW:rad}), $\Sigma_x$ is homeomorphic to the topological suspension $S(\partial B(\zeta,\pi/2))$ over $\partial B(\zeta,\pi/2)\subset\Sigma_x$, where $\zeta$ is one of the poles of the suspension (the other pole is a unique direction farthest from $\zeta$).
Since $\zeta$ is close enough to $p'_x$, we see that $\partial B(\zeta,\pi/2)$ is homeomorphic to $\partial B(p'_x,\pi/2)$.
More precisely, this follows from the fact that the distance functions $d(\zeta,\cdot)$ and $d(p'_x,\cdot)$ are both regular in the same direction to $\eta$ on the $\pi/5$-neighborhood of $\partial B(\zeta,\pi/2)$.
Using this fact, one can construct a noncritical function $\phi$ in the sense of \cite[3.1]{Per:alex} of the form
\[\phi=\min\{\phi_1\circ d(\zeta,\cdot),\phi_2\circ d(p'_x,\cdot)\}\]
where $\phi_i$ ($i=1,2$) are Lipschitz functions chosen so that $\phi$ coincides with $d(\zeta,\cdot)$ around $\partial B(\zeta,\pi/2-\pi/10)$ and with $d(p'_x,\cdot)$ around $\partial B(\zeta,\pi/2+\pi/10)$.
Then one can apply the fibration theorem \cite[1.4.1]{Per:alex} to $\phi$ to show that $\partial B(\zeta,\pi/2-\pi/10)$ and $\partial B(p'_x,\pi/2+\pi/10)$ are homeomorphic.
By the fibration theorem again, these are homeomorphic to $\partial B(\zeta,\pi/2)$ and $\partial B(p'_x,\pi/2)$, respectively.
Finally, the stability theorem \cite[4.3]{Per:alex} implies that $\Sigma$ is homeomorphic to $\partial B(p'_x,\pi/2)$, provided $\varepsilon$ is small enough.
Therefore $\Sigma_x$ is homeomorphic to $S(\Sigma)$.

Moreover, as explained in Step 1, any point other than the poles of the suspension is not an extremal point, and thus a manifold point.
Therefore, $\Sigma_x$ is a topological manifold if and only if $\Sigma$ is simply-connected (Property (3)).

Since $r$ is small enough, it follows from the stability theorem that $\partial B(p,r)$ is homeomorphic to $\Sigma_p$.
In particular, $\partial B(p,r)$ is a homology manifold with isolated nonmanifold points (Property (2)).
Since $\varepsilon$ is small enough, it follows from the fibration theorem that $\partial B(p,r)\cap B(x,\varepsilon)$ is homeomorphic to the open cone $C(\Sigma)$ (more precisely, the fibration theorem implies that $\partial B(p,r)\cap\bar B(x,t)\setminus B(x,s)$ is homeomorphic to $\Sigma\times[s,t]$ for any $0<s<t<\varepsilon$, and by gluing such homeomorphisms we obtain the desired cone structure).
Thus $\Sigma$ is simply-connected if and only if $x$ is a manifold point of $\partial B(p,r)$ (the ``only if'' part follows from \cite{CBL}, \cite{BDVW}; see also \cite[6.2]{LN:top} and note that $\dim\partial B(p,r)\ge 4$).
In summary,
\begin{align*}
&\text{$\Sigma_x$ is a topological sphere (equivalently, a topological manifold)}\\
\iff&\text{$\Sigma=\partial B(p,r)\cap\partial B(x,\varepsilon)$ is simply-connected}\\
\iff&\text{$x$ is a manifold point of $\partial B(p,r)$}.
\end{align*}
Since the rescaled sphere $r^{-1}\partial B(p,r)$ converges to $\Sigma_p$ as $r\to0$ and the stability homeomorphism is a Gromov--Hausdorff approximation, the claim follows.
\end{proof}

Claim \ref{clm:link} and Property (2) imply that $\Sigma_pE$ is an extremal subset of $\Sigma_p$.
Claim \ref{clm:link} also shows that $E$ is closed and has no isolated points.
By Lemma \ref{lem:inf} and the additional condition proved in Step 1, we conclude that $E$ is an extremal subset.
This completes the proof.
\end{proof}

\begin{rem}\label{rem:hom}
More generally, Theorem \ref{thm:link'} holds for any Alexandrov homology manifold $X$: the conclusion in this setting is that $E$ is an extremal subset of dimension $\le 1$ (in particular, $E$ contains the isolated nonmanifold points of $X$).
Let us trace the above proof, checking for required modifications.

The dimension $\le4$ case follows from Properties (1)--(4), so we may assume dimension $\ge 5$ as before.
Step 1 of the proof of Theorem \ref{thm:link'} works exactly the same way.
Step 2 requires the following minor modifications.
In the proof of Claim \ref{clm:link}, we used the assumption that $X$ is a topological manifold to ensure that $\Sigma_x$ is simply-connected (Property (3)), which implied that $\Sigma_x$ is a topological sphere if and only if it is a topological manifold.
However, in the proof of Claim \ref{clm:link}, $\Sigma_x$ is a topological suspension over $\Sigma$, where $\Sigma$ is connected since $\partial B(p,r)$ is a homology manifold.
Therefore $\Sigma_x$ is automatically simply-connected.
Another difference is that, in the last paragraph, Claim \ref{clm:link} may not imply that $E$ has no isolated points.
We have to check that if $E$ has an isolated point $p$, then it is an extremal point.
This is proved as follows.
Since $p$ is isolated in $E$, Claim \ref{clm:link} implies that $\Sigma_p$ is a topological manifold.
Thus $\Sigma_p$ is not simply-connected (otherwise, it is a topological sphere).
By Property (3), $p$ is an extremal point.
\end{rem}

\begin{rem}
Here we give a more geometric interpretation of the ``only if'' part of Claim \ref{clm:link}.
Let $\xi\in\Sigma_p$ and $r_i\to 0$.
Consider the pointed Gromov--Hausdorff convergence $({r_i}^{-1}X,p)\to(T_p,o)$, where $(T_p,o)$ is the tangent cone of $X$ at $p$ based at the vertex.
There is a canonical way of lifting $\xi$ to $p_i\in{r_i^{-1}}X$ so that the distance function $d(p_i,\cdot)$ is uniformly regular around $p_i$ independent of $i$, that is,
\[|\nabla d(p_i,\cdot)|>\varepsilon\]
on $r_i^{-1}\bar B(p_i, r_i\varepsilon)\setminus\{p_i\}$, where $\varepsilon>0$ is independent of $i$ (see \cite{Pet:sc} for the definition of the gradient).
For the proof, see \cite[3.2]{Y:ess}, \cite[1.2]{MY:stab} (actually a similar argument is included in the proof of the stability theorem used above).
This, together with the fibration and stability theorems, implies that $\Sigma_{p_i}$ is homeomorphic to $r_i^{-1}\partial B(p_i,r_i\varepsilon)$.
By the fibration and stability theorems again, $r_i^{-1}\partial B(p_i,r_i\varepsilon)$ is homeomorphic to $\partial B(\xi,\varepsilon)\subset T_p$, which is homeomorphic to the space of directions of $T_p$ at $\xi$, provided $\varepsilon$ is small enough (cf.\ \cite[5.1]{K:stab}).
This space of direction is isometric to the spherical suspension $S(\Sigma_\xi)$, where $\Sigma_\xi$ is the space of directions of $\Sigma_p$ at $\xi$.
Therefore,
\begin{align*}
&\text{$\xi$ is a manifold point of $\Sigma_p$}\\
\iff&\text{$\Sigma_\xi$ is simply-connected}\\
\iff&\text{$S(\Sigma_\xi)$, equivalently $\Sigma_{p_i}$, is a topological sphere}.
\end{align*}
This shows the ``only if'' part of Claim \ref{clm:link}.
It might also be possible to prove the ``if'' part by modifying this lifting argument, but we do not pursue it here.
\end{rem}

As an application of Theorem \ref{thm:link'}, we obtain a new sphere theorem for the link of an Alexandrov homology manifold with positive curvature.

\begin{cor}\label{cor:link}
Let $\Sigma$ be an Alexandrov homology manifold with curvature $\ge 1$.
If $\Sigma$ contains four points with pairwise distance $>\pi/2$, then every space of directions is homeomorphic to a sphere.
\end{cor}

Here $\Sigma$ itself is also a topological sphere, in particular, a topological manifold; see Remark \ref{rem:three}(2).
Hence Corollary \ref{cor:link} is an immediate consequence of Theorem \ref{thm:link'} and the $k=1$ case of the following general lemma.
By convention, the empty set has dimension $-1$.

\begin{lem}\label{lem:reg}
Let $\Sigma$ be an Alexandrov space with curvature $\ge 1$.
Suppose $\Sigma$ contains a $k$-dimensional extremal subset $F$, where $-1\le k\le\dim\Sigma$, satisfying the additional conditions as in Lemma \ref{lem:inf} when $k=-1,0$.
Then $\Sigma$ does not contain $k+3$ points with pairwise distance $>\pi/2$. 
\end{lem}

Note that the additional conditions above do not matter when proving Corollary \ref{cor:link}, because the extremal subset obtained from Theorem \ref{thm:link'} has dimension $1$.

\begin{proof}[Proof of Lemma \ref{lem:reg}]
Let $C(\Sigma)$ be the Euclidean cone over $\Sigma$.
Then the subcone $C(F)$ is a $(k+1)$-dimensional extremal subset (this follows, for example, Lemma \ref{lem:inf}).
If $\Sigma$ contains $k+3$ points with pairwise distance $>\pi/2$, then there exists a regular map at the vertex of $C(\Sigma)$ to $\mathbb R^{k+2}$.
This contradicts Remark \ref{rem:dim}.
\end{proof}

\begin{rem}\label{rem:opt}
Corollary \ref{cor:link} is optimal in the sense that the number of discrete points and the inequality for the distance in the assumption cannot be relaxed.
Indeed, the double suspension of a homology sphere as in Example \ref{ex:ds} contains four points with pairwise distance $\ge\pi/2$.
\end{rem}

\begin{rem}
In general, if an Alexandrov space $\Sigma$ with curvature $\ge1$ contains $k+1$ points with pairwise distance $>\pi/2$, then $\Sigma$ is homeomorphic to a $k$-fold suspension (\cite[Theorem C]{GW:rad}).
In the maximal case, i.e., $k=\dim\Sigma+1$, we obtain that $\Sigma$ is bi-Lipschitz homeomorphic to the standard sphere (cf.\ \cite[3.2]{Per:alex}, \cite{Per:mor}).
In the Riemannian case, the existence of such $\pi/2$-discrete points gives some restriction on the differentiable structure (\cite{GW:exo}).
\end{rem}

\begin{rem}\label{rem:three}
Let us recall what can be said if the number of $\pi/2$-discrete points is smaller (for the larger case, see Corollary \ref{cor:poly} and Remark \ref{rem:conj}).
An Alexandrov homology manifold $\Sigma$ with curvature $\ge 1$ is a topological sphere if one of the following holds:
\begin{enumerate}
\item $\Sigma$ contains no nonmanifold points and has diameter $>\pi/2$; or
\item $\Sigma$ contains three points with pairwise distance $>\pi/2$.
\end{enumerate}
The first is a generalization of the diameter sphere theorem of Grove--Shiohama \cite{GS}.
This follows from the suspension theorem and the basic fact that a sphere is the only manifold that is topologically a suspension (cf.\ \cite[8.2]{LN:top}).
The second follows from the first, Property (2), and the following observation: under the assumption of (2), $\Sigma$ contains no extremal points.
Indeed, for any $x\in\Sigma$, two of the three $\pi/2$-discrete points must lie inside or outside of the $\pi/2$-neighborhood of $x$, and thus they make an obtuse angle at $x$.
Note that, as with Remark \ref{rem:opt}, the above assumptions are optimal.
\end{rem}

\section{Polyhedral case}\label{sec:poly}

Finally we prove Theorem \ref{thm:poly}.
We first introduce polyhedral spaces.
See \cite[Ch.\ 12]{AKP:found} (cf.\ \cite[\S3.2]{BBI}) for more details.
However, unlike \cite[12.1]{AKP:found}, we will only assume local compactness for polyhedral spaces instead of global compactness.

Let $\kappa\in\mathbb R$.
The \textit{model space} of curvature $\kappa$ refers to a sphere, Euclidean space, or a hyperbolic space of constant curvature $\kappa$.
A complete geodesic space $X$ is called a \textit{polyhedral space} if it admits a triangulation (i.e., a homeomorphism to a locally finite simplicial complex) such that each simplex in $X$ is isometric to a simplex in the model space of curvature $\kappa$, for some fixed $\kappa$.
The \textit{triangulation} of a polyhedral space means the triangulation in this definition.
The cases $\kappa=1,0,-1$ are called, respectively, a \textit{spherical, Euclidean, hyperbolic} polyhedral space.
The \textit{dimension} of a polyhedral space means the maximal dimension of a simplex (we only consider finite-dimensional polyhedral spaces).
A characterization of an Alexandrov metric on a polyhedral space can be found in \cite[12.5]{AKP:found} (cf.\ \cite[4.2.14]{BBI}).

We first recall the basic properties of general (i.e., not necessarily Alexandrov) polyhedral metrics.
Let $X$ be a polyhedral space.
For a simplex $\sigma$ in $X$, the \textit{link} around $\sigma$, denoted by $L(\sigma)$, is the subcomplex of $X$ formed by all the simplexes $\sigma'$ such that $\sigma$ and $\sigma'$ are disjoint but they are faces of a common simplex in $X$.
The link $L(\sigma)$ is identified with the set of normal directions to $\sigma$ at some interior point of $\sigma$.
Since the latter carries the angle metric, this makes $L(\sigma)$ a spherical polyhedral space.

Let $p\in X$.
The \textit{space of directions} $\Sigma_pX$ of $X$ at $p$ is defined in the same way as in Alexandrov geometry: the set of directions equipped with the angle metric (in this section the notation emphasizes the ambient space $X$).
Note that $\Sigma_pX$ is a spherical polyhedral space of lower dimension.
More precisely, if $p$ is contained in the interior of a $k$-dimensional simplex $\sigma$, then we have
\begin{equation}\label{eq:link}
\Sigma_pX=S^k(L(\sigma)),
\end{equation}
where $S^k(\cdot)$ denotes the $k$-fold spherical suspension.
Clearly a small neighborhood of $p$ is isometric to that of the vertex of the \textit{$\kappa$-cone} over $\Sigma_pX$ (see \cite[\S11.C]{AKP:found} or \cite[10.2.2]{BBI} for the definition of $\kappa$-cone).

By definition, $\Sigma$ is a \textit{$k$-th iterated space of directions} of $X$ at $p$ if there exists a sequence $p=\xi_1,\dots,\xi_k$ such that
\[\Sigma=\Sigma_{\xi_k}\cdots\Sigma_{\xi_1} X,\quad \xi_i\in\Sigma_{\xi_{i-1}}\cdots\Sigma_{\xi_1}X\]
for any $2\le i\le k$.
We denote this simply by
\[\Sigma=\Sigma^kX\]
without specifying the sequence $\xi_1,\dots,\xi_k$.

We say that a polyhedral space $X$ has \textit{spherical links} if the link around every simplex in $X$ is homeomorphic to a sphere.
If $X$ has spherical links, then $X$ is a topological manifold.
Furthermore, the link $L(\sigma)$ of a simplex $\sigma$ in $X$ also has spherical links.
The following lemma describes the relation between Conjecture \ref{conj:pl} and Theorem \ref{thm:poly}.

\begin{lem}\label{lem:pl}
Let $X$ be a polyhedral space.
Then $X$ has spherical links if and only if every iterated space of directions is homeomorphic to a sphere.
\end{lem}

\begin{proof}
The ``if'' part follows by taking iterated spaces of directions of both sides of Equation \eqref{eq:link}.
The ``only if'' part is proved by induction on dimension.
If $X$ has spherical links, then Equation \eqref{eq:link} shows that every space of directions is homeomorphic to a sphere.
Since the link also has spherical links, the inductive assumption and Equation \eqref{eq:link} show that every $k$-th iterated space of directions ($k\ge 2$) is homeomorphic to a sphere.
\end{proof}

Finally, we give a direct characterization of an extremal subset in a polyhedral Alexandrov space.
Let $X$ be a polyhedral Alexandrov space.
Then a subset $E$ of $X$ is extremal (in the sense of Definition \ref{dfn:ex}) if and only if
\begin{itemize}
\item $E$ is a subcomplex of $X$; and
\item for every simplex $\sigma$ in $E$, if $\xi\notin B(L_E(\sigma),\pi/2)$, then $L(\sigma)\subset\bar B(\xi,\pi/2)$.
\end{itemize}
Here $L_E(\sigma)$ denotes the link of $\sigma$ in the subcomplex $E$, which is a subcomplex of $L(\sigma)$.
The proof is straightforward and left to the reader (we will not use it).

We now prove Theorem \ref{thm:poly}.
We show the following more rigid result.
Compare with Theorem \ref{thm:link'} and Remark \ref{rem:hom}.

\begin{thm}\label{thm:poly'}
Let $X$ be a polyhedral Alexandrov homology $n$-manifold.
For any integer $1\le k\le n$, let $E_k$ be the set of points $p\in X$ such that some $k$-th iterated space of directions at $p$ is not homeomorphic to $S^{n-k}$.
If $E_k$ is nonempty, then it is an extremal subset of $X$ that is a subcomplex of $X$ only with $k$- or $(k-1)$-dimensional maximal simplexes.
\end{thm}

By Properties (1) and (4) in Section \ref{sec:link}, $E_k$ is empty for $k\ge n-2$, so we may assume $k\le n-3$.
Moreover, we have already proved in Theorem \ref{thm:link'} and Remark \ref{rem:hom} that $E_1$ is an extremal subset of dimension $\le 1$.
In the polyhedral case, we also see that $E_1$ is a subcomplex of $X$, since the space of directions is invariant in the interior of a simplex as in Equation \eqref{eq:link}.
Therefore we may assume $k\ge 2$ (and hence $n\ge 5$).

We prove Theorem \ref{thm:poly'} by induction on dimension (more precisely, on $k$).
Recall that, if $X$ is a polyhedral Alexandrov homology manifold, then $\Sigma_pX$ is again an Alexandrov polyhedral homology manifold with the homology of a sphere.
The key observation is the following.
Compare with Claim \ref{clm:link}.

\begin{clm}\label{clm:poly}
Let $X$ be as in Theorem \ref{thm:poly'} and $p\in X$.
Then we have
\[\Sigma_pE_k(X)=E_{k-1}(\Sigma_pX)\]
for any $2\le k\le n$, where $E_k(\cdot)$ denotes the singular set defined in Theorem \ref{thm:poly'}.
The equality remains valid even if either side is empty.
\end{clm}

\begin{rem}
Claim \ref{clm:link} showed that
\[\Sigma_pE_1(X)=E_0(\Sigma_pX),\]
where $E_0(\cdot)$ denotes the set of nonmanifold points.
The following proof of Claim \ref{clm:poly} also works for $k=1$, which is indeed a baby version of the proof of Claim \ref{clm:link}.
Theorem \ref{thm:poly'} also holds for $E_0$ (Property (2) in Section \ref{sec:link}). 
\end{rem}

\begin{proof}[Proof of Claim \ref{clm:poly}]
The proof only uses the properties of polyhedral spaces and Edwards' characterization of topological manifolds \cite[\S8]{E:icm} (cf.\ Properties (1)--(4) in Section \ref{sec:link}).
We do not rely on any properties of Alexandrov spaces.

Let $U$ be a neighborhood of $p$ that is isometric to a neighborhood of the vertex of the $\kappa$-cone over $\Sigma_pX$.
Let $x\in U$ be a point in the direction $\xi\in\Sigma_pX$.
Then we have
\begin{equation}\label{eq:poly1}
\xi\in\Sigma_pE_k(X)\iff x\in E_k(X)\iff\Sigma^{k-1}\Sigma_xX\ncong S^{n-k},
\end{equation}
where $\cong$ denotes a homeomorphism.

Since $\Sigma_xX=S(\Sigma_\xi\Sigma_pX)$, we have the following two possibilities:
\begin{equation*}
\Sigma^{k-1}\Sigma_xX=
\begin{cases}
\hfil\Sigma^{k-2}\Sigma_\xi\Sigma_pX\\
S(\Sigma^{k-1}\Sigma_\xi\Sigma_pX)
\end{cases}
\end{equation*}
(consider the $k=2$ case and apply it repeatedly).
In the first case, it is clear that $\Sigma^{k-1}\Sigma_xX$ is not homeomorphic to $S^{n-k}$ if and only if $\xi\in E_{k-1}(\Sigma_pX)$.
In the second case, $\Sigma^{k-1}\Sigma_xX$ is not homeomorphic to $S^{n-k}$ if and only if $\Sigma^{k-1}\Sigma_\xi\Sigma_pX$ is not simply-connected, by Edwards' theorem \cite[\S8]{E:icm} (or Property (3) in Section \ref{sec:link}).
The latter is equivalent to that $\Sigma^{k-2}\Sigma_\xi\Sigma_pX$ is not a manifold, i.e., $\xi\in E_{k-1}(\Sigma_pX)$.
Therefore, in either case,
\begin{equation}\label{eq:poly2}
\Sigma^{k-1}\Sigma_xX\ncong S^{n-k}\iff\xi\in E_{k-1}(\Sigma_pX).
\end{equation}
Combining \eqref{eq:poly1} and \eqref{eq:poly2}, we obtain the conclusion.
\end{proof}

\begin{proof}[Proof of Theorem \ref{thm:poly'}]
We prove this by induction on $k$.
The base case, $k=1$, was already proved in Theorem \ref{thm:link'} and Remark \ref{rem:hom}.
Suppose $k\ge 2$.

We first show that $E_k(X)$ is a subcomplex (in particular, a closed subset of $X$).
This follows from Claim \ref{clm:poly}.
If $p$ is contained in the closure of $E_k(X)$, then Claim \ref{clm:poly} implies that $E_{k-1}(\Sigma_pX)$ is nonempty.
This means that $p\in E_k(X)$, i.e., $E_k(X)$ is closed.
If the interior of a simplex $\sigma$ intersects $E_k(X)$, then clearly every point in the interior is contained in $E_k(X)$.
Since $E_k(X)$ is closed, the boundary of $\sigma$ is contained in $E_k(X)$, i.e., $E_k(X)$ is a subcomplex.

By the inductive assumption and Claim \ref{clm:poly}, for any $p\in E_k(X)$, the space of directions $\Sigma_pE_k(X)$ is an extremal subset of $\Sigma_pX$ that is a subcomplex of $\Sigma_pX$ only with $(k-1)$- or $(k-2)$-dimensional maximal simplexes.
This implies that $E_k(X)$ only have $k$- or $(k-1)$-dimensional maximal simplexes.
Furthermore, by Lemma \ref{lem:inf}, $E_k(X)$ is an extremal subset of $X$.
Note that the additional conditions in Lemma \ref{lem:inf} are satisfied.
Indeed, Claim \ref{clm:poly} implies that $\Sigma_pE_k(X)$ is nonempty (since $k\ge 2$), and if $\Sigma_pE_k(X)$ consists of one element (when $k=2$), the same argument as in Step 1 of the proof of Claim \ref{clm:link} shows that $\diam\Sigma_pE_k(X)\le\pi/2$.
This completes the proof.
\end{proof}

The following corollary immediately follows from Theorem \ref{thm:poly'} and Lemma \ref{lem:reg}.
Compare with Corollary \ref{cor:link}.

\begin{cor}\label{cor:poly}
Let $\Sigma$ be a polyhedral Alexandrov homology manifold with curvature $\ge 1$.
Suppose $\Sigma$ contains $k+3$ points with pairwise distance $>\pi/2$.
Then every $k$-th iterated space of directions is homeomorphic to a sphere.
\end{cor}

\begin{rem}\label{rem:conj}
The above proof of Theorem \ref{thm:poly'} (except for the part about polyhedral structures) only uses general properties of Alexandrov spaces.
Therefore the proof of Conjecture \ref{conj:pl} in the general case reduces to the proof of Claim \ref{clm:poly}.
Similarly, if we obtain Theorem \ref{thm:poly'} for nonpolyhedral spaces, then it immediately implies Corollary \ref{cor:poly} in the nonpolyhedral case.
\end{rem}

Finally, we remark that the absence of extremal subsets, as in Conjecture \ref{conj:pl}, is not inherited to the space of directions (if it were, the proof of Conjecture \ref{conj:pl} would be straightforward by induction on dimension).

\begin{ex}\label{ex:ab}
Let $P$ be a truncated cube obtained by cutting off all the corners of a solid cube by planes.
Let $p\in P$  a midpoint of an edge of the original cube.
Consider the double $D(P)$ of $P$, i.e., the gluing of two copies of $P$ along the boundary.
Then $D(P)$ contains no proper extremal subsets.
On the other hand, the space of directions $\Sigma_pD(P)$ is isometric to the double of the quadrant of the unit sphere, which contains a $0$-dimensional extremal subset.
A similar example can be obtained by taking the intersection of the solid cube in $\mathbb R^3$ with edge length $2$ and the concentric Euclidean ball of radius $\sqrt 2$ and doubling it.
In this case the resulting space has no proper extremal subsets even ``locally'' around the midpoint of the original edge.
\end{ex}

\appendix

\section{}\label{sec:app}

In this appendix, we give an example of a primitive extremal subset of codimension $2$ whose induced intrinsic metric is not an Alexandrov metric with a lower curvature bound.
The construction is very simple, but to the best of the author's knowledge, no such examples can be found in the literature.

Let us briefly review the history.
One of the long-standing problems in Alexandrov geometry is the \textit{boundary conjecture} --- the boundary of an Alexandrov space equipped with the induced intrinsic metric would have the same lower curvature bound as the ambient space.
This is true for a convex hypersurface in a Riemannian manifold (\cite{AKP:opt}), but is still open even for a convex hypersurface in a \textit{smoothable} Alexandrov space, i.e., a noncollapsing limit of Riemannian manifolds with sectional curvature bounded below (\cite[p.3871]{LP:curv}).
Recall that the boundary is a codimension $1$ extremal subset.

As a generalization of the boundary conjecture, Perelman--Petrunin \cite[6.1]{PP:ext} asked whether the induced intrinsic metric of a primitive extremal subset admits a lower curvature bound or not.
In \cite[\S1]{Pet:app}, Petrunin constructed an example of a primitive extremal subset of codimension $3$ (and thus $\ge 3$) that is not an Alexandrov space (cf.\ \cite[4.1]{Per:icm}).
However, other than the boundary case, the codimension $2$ case remained.
In \cite[9.1.2]{Pet:sc}, Petrunin conjectured that a primitive extremal subset of codimension $2$ might be an Alexandrov space.

The following example provides a negative answer to this conjecture.

\begin{ex}
Let $X$ be a convex subset of $\mathbb R^4$ defined by
\[X:=\left\{(x,y,z,w)\in\mathbb R^4\mid x\ge 0,\ y\ge 0,\ z\ge 0\right\}.\]
Suppose $\mathbb Z$ acts isometrically on $X$ by the cyclic permutation on the $(x,y,z)$-coordinate and by the translation on the $w$-coordinate.
More precisely, the action is generated by
\[\sigma:(x,y,z,w)\mapsto(y,z,x,w+1).\]
Then the quotient space $X/\mathbb Z$ is an Alexandrov space with nonnegative curvature.
Let $\pi:X\to X/\mathbb Z$ be the projection.

The $xw$-, $yw$-, and $zw$-planes are primitive extremal subsets of $X$.
In particular, their union $E$ is an extremal subset of $X$, but not primitive.
The projection $\pi(E)$ is an extremal subset of codimension $2$ in $X/\mathbb Z$ (\cite[4.1]{PP:ext}).
Furthermore, $\pi(E)$ is primitive because it can no longer be divided into each plane, due to the effect of the group action.
However, $\pi(E)$ does not admit any Alexandrov metric with a lower curvature bound as its intrinsic geodesics branch.
\end{ex}

An extremal subset is called \textit{minimal} if it contains no proper extremal subsets.
Clearly a minimal extremal subset is primitive.
Note that any minimal extremal subset must be a topological manifold, because the set of topological singularities, if nonempty, is a proper extremal subset.
In particular, the above counterexample is not minimal: the projection of the $w$-axis is a proper extremal subset of $\pi(E)$.

The following possibility still remains.

\begin{ques}
Is there a minimal extremal subset of codimension $2$ that is not an Alexandrov space with respect to the induced intrinsic metric?
\end{ques}

\begin{rem}
Petrunin's counterexample in \cite[\S1]{Pet:app} is also not minimal: the vertex of the cone is an extremal point.
However, it is possible to slightly modify it so that the vertex becomes a nonextremal point.
\end{rem}

\end{document}